\documentclass[11pt, reqno]{amsart}

\usepackage[margin=1in]{geometry}  % set the margins to 1in on all sides
\usepackage{graphicx}              % to include figures
\usepackage{amsmath}               % great math stuff
\usepackage{amsfonts}              % for blackboard bold, etc
\usepackage{amsthm}      % better theorem environments
\usepackage{mathtools}          
\usepackage{color}
\usepackage{tabulary}
\usepackage{caption}
\usepackage{subcaption}
\usepackage{amssymb}
\usepackage{todonotes}
\theoremstyle{plain}
\usepackage{mathptmx}
\usepackage{tikz-cd} 
\usepackage{mathrsfs}
\usepackage{verbatim}
\usepackage{tikz}
\usepackage{enumitem}
\usepackage{hyperref}
\newtheorem{thm}{Theorem}[section]
\newtheorem{lem}[thm]{Lemma}
\newtheorem{prop}[thm]{Proposition}
\newtheorem{cor}[thm]{Corollary}

\theoremstyle{definition}
\newtheorem{defn}[thm]{Definition}
\newtheorem{rmk}[thm]{Remark}

\newtheorem{fact}[thm]{Fact}

\DeclareMathOperator{\Id}{Id}

\DeclareMathOperator{\Aut}{Aut}

\newcommand{\ZZ}{\mathbb{Z}}      % for Integers
 
\newcommand{\CC}{\mathbb{C}}      % for Real numbers
\DeclareMathOperator{\PSL}{PSL}

\DeclareMathOperator{\BS}{BS}

\numberwithin{equation}{section}

\begin{document}

\title{Bi-orderability and generalized torsion elements from the perspective of profinite properties}
\author{Wonyong Jang}
\address{Department of Mathematical Sciences, KAIST,
291 Daehak-ro, Yuseong-gu, 
34141 Daejeon, South Korea}
\email{jangwy@kaist.ac.kr}

\author{Junseok Kim}
\address{Department of Mathematics,
        Technion - Israel Institute of Technology, 
        Technion City, Haifa, Israel, 3200003}
\email{jsk8818@campus.technion.ac.il}

\maketitle

%%%%%%%%%%%%%%%%%%%%%%%%%%%%%%%%%%%%%%%%%%%%%%%%%%%%%%%%%%%%%%%%%%%%%%

\begin{abstract}
Using fiber products, we construct bi-orderable groups from left-orderable groups. 
As an application, we show that bi-orderability is not a profinite property, answering a question of Piwek and Wykowski negatively.
We also show that the existence of a generalized torsion element is not a profinite property.
\end{abstract}

\vspace{0.6cm} 

% \noindent
% \textbf{AMS Classification numbers (2020)} Primary: 06F15, 20E18, 20F60, Secondary: 20E06, 20E22, 20E26

% 06F15 Ordered groups
% 20E18 Limits, profinite groups
% 20F60 Ordered groups (group-theoretic aspects)
% 20E06 Free products of groups, free products with amalgamation, Higman-Neumann-Neumann extensions, and generalizations
% 20E22 Extensions, wreath products, and other compositions of groups
% 20E26 Residual properties and generalizations; residually finite groups

% 20J05 Homological methods in group theory?

% Primary: 20F65, 20F67, Secondary: 20F69

\vspace{0.6cm} 

%%%%%%%%%%%%%%%%%%%%%%%%%%%%%%%%%%%%%%%%%%%%%%%%%%%%%%%%%%%%%%%%%%%%%%

%%%%%%%%%%%%%%%%%%%%%%%%%%%%%%%%%%%%%%%%%%%%%%%%%%%%%%%%%%%%%%%%%%%%%%

\section{Introduction} \label{sec:intro}

 Understanding a finitely generated group by its finite quotients has been a topic of much interest in low-dimensional geometry and topology as well as geometric group theory. 
 Amongst finitely generated residually finite groups, a small number of groups can be determined up to isomorphism by their sets of finite quotient groups.
 This list includes 
 some triangle groups \cite{MR4386043},
 some arithmetic lattices in $\PSL(2,\CC)$ \cite{bridson2020absolute},
 affine Coxeter groups (\cite{paolini2024profinite} and \cite{corson2024profinite}), the fundamental group of some closed fibered hyperbolic 3-manifolds \cite{MR4831034}, lamplighter groups (\cite{blachar2024profinite} and \cite{wykowski2025profinite1}), 
 the solvable Baumslag-Solitar groups $\BS(1,n)$ \cite{wykowski2025profinite1}, and finitely generated free metabelian groups \cite{wykowski2025profinite2}.

 Despite the above positive results, there are numerous examples of finitely generated residually finite groups that share the same collection of finite quotients but are not isomorphic.
 Baumslag proved that there are virtually cyclic groups with the same finite quotients but not isomorphic \cite{baumslag1974residually}. Pickel proved the analogous statement for metabelian groups \cite{pickel1974metabelian}.
 In addition, similar phenomena --- namely, that the set of finite quotients cannot determine the group structure --- 
  can be found in the class of finitely presented groups (\cite{MR2119723} and \cite{MR3531665}), 
 in direct products of free groups (\cite{MR2119723} and \cite{corson2023higman}), and in Coxeter groups \cite{corson2023higman}.
% Moreover, Pyber constructed uncountably many finitely generated groups that have the same finite quotients \cite{pyber2004groups}, and there exists a finitely presented group that can be detected by finite quotients amongst finitely presented groups, but not amongst finitely generated groups \cite{bridson2023absolute}.
Moreover, Pyber constructed uncountably many finitely generated groups that have the same finite quotients \cite{pyber2004groups}.
Bridson, Reid, and Spitler constructed a finitely presented group whose isomorphism type can be detected by finite quotients amongst finitely presented groups, but not amongst finitely generated groups \cite{bridson2023absolute}.
 For further details on this rigidity and related pathological examples, we refer to \cite{bridson2025relatively}, \cite{bridson2025chasing}, and references therein.

 Amongst $3$-manifold groups, many properties can be detected by finite quotients of the fundamental group. 
 For example, the 3-dimensional geometry in the sense of Thurston \cite{MR3608716} or the JSJ decomposition \cite{MR3774902} can be detected in this way. 
 Some progress has been made in detecting geometric structures in \mbox{dimension~4}~\cite{MR4458587}.
 
 Similarly, one can ask whether certain group properties are preserved amongst groups with the same finite quotients.
 To describe this notion precisely, we need the concept of profinite completion.
We refer to Subsection~\ref{Sebsec-PP} for the definition.
%We point out that for finitely generated groups $G$ and $H$, they have the isomorphic profinite completion if and only if they have the same finite quotients up to isomorphism (\cite{dixon1982profinite} and \cite{nikolov2007finitely}).
Recall that two finitely generated groups $G$ and $H$ have isomorphic profinite completions if and only if they have the same set of finite quotients (see \cite{dixon1982profinite} and \cite{nikolov2007finitely}).

Accordingly, we can consider a group property detected by finite quotients as follows.
For two finitely generated residually finite groups $G, \ H$ with the same profinite completions $\widehat{G} \cong \widehat{H}$, if $G$ has a group property $\mathcal{P}$, then so does $H$, and vice versa. 
 Such a property is called a \emph{profinite property}, and it is being actively studied.
 The abelianization and groups satisfying a law \cite{piwek2024profinite} are typical examples of a profinite property.
 Unfortunately, numerous properties are known to be not profinite properties.
 For example, amenability \cite{MR4641368}, bounded cohomology \cite{echtler2024bounded}, 
finiteness properties \cite{MR3161766}, Kazhdan’s property (T) \cite{MR2914858} and property FA \cite{cheetham2024property} are not profinite properties.
 In \cite{MR4732948}, Bridson proved that many properties are not profinite properties at once, including being torsion-free, being locally indicable, and left-orderability, etc.

 In this paper, we mainly concentrate on bi-orderability and prove the following.

 \begin{thm}[Corollary \ref{Thm-BOPP}]
 Bi-orderability is not a profinite property.
\end{thm}

 %As mentioned in the abstract, this answers negatively the question of Piwek and Wykowski in the problem list that came from participants of the workshop ``Profinite Rigidity" held in Madrid from June 26 to June 30, 2023.
 As mentioned in the abstract, this provides a negative answer to the question posed by Piwek and Wykowski in the problem list compiled during the workshop ``Profinite Rigidity" held in Madrid in June 2023.
 
 It is worth noting that the above theorem does not follow directly from the results in \cite{MR4732948}.
 More precisely, for any finitely presented residually finite group $\Gamma$, Bridson constructed a free-by-free group $M_{\Gamma}=F_{\infty}\rtimes F_4$ with an embedding $M_{\Gamma} \hookrightarrow (F_4 * \Gamma) \times F_4$ such that $ \widehat{M_{\Gamma}} \cong \widehat{F_4 * \Gamma} \times \widehat{F_4}.$
 However, a free-by-free group is not necessarily bi-orderable. For further details, we refer to Remark~\ref{free-by-free BO}.
 Thus, while we use the technique introduced in that paper, we slightly modify two epimorphisms to construct a bi-orderable fiber product from a left-orderable group.

\begin{thm}[Theorem \ref{Thm-BO}]
 Let $G$ be a non-trivial finitely generated left-orderable group. Say $G$ is generated by $n$ elements, and consider the following two epimorphisms 
 $$ \pi_1 : F_n*G \twoheadrightarrow G , \qquad \pi_2 : F_n \twoheadrightarrow G $$ where $\pi_2$ is the natural epimorphism, $\pi_1 |_{F_n} = \pi_2$, and $\pi_1|_{G}=\Id_G$, the identity map on $G$.
 Then the corresponding fiber product $P$ is bi-orderable.
\end{thm}

The idea of the proof of Corollary~\ref{Thm-BOPP} is as follows.
We first choose a pair $(G,S)$ as in Lemma~\ref{Core-Lem}. 
The essential properties are the following: $G$ is a finitely presented, left-orderable group with $\widehat G=1$ and $H_2(G;\mathbb Z)=0$, and $S$ is a finitely generated, residually finite subgroup of $G$ that is not bi-orderable.
The left-orderability of $G$ allows us to apply Theorem~\ref{Thm-BO} to obtain a bi-orderable fiber product $P<(F_n*G)\times F_n$.
The conditions $\widehat G=1$ and $H_2(G;\mathbb Z)=0$ together with finite presentability of $G$ are used to apply Proposition~\ref{Fiber_profinite_completion}.

By restricting the epimorphism $\pi_1$ to $q_1=\pi_1|_{F_n * S} : F_n * S \to G$, we construct a new fiber product $Q<(F_n * S) \times F_n$. Since $Q$ is a subgroup of the bi-orderable group $P$, $Q$ is itself bi-orderable. 
Proposition~\ref{Fiber_profinite_completion} then yields an isomorphism of profinite completions
$$ \widehat Q\cong \widehat{F_n*S}\times \widehat{F_n}. $$ Since $(F_n * S) \times F_n$ is not bi-orderable (as it contains $S$), this isomorphism proves that bi-orderability is not a profinite property.

 We say that $g \in G$ is a \emph{generalized torsion element} if there exist finitely many $g_1, \ldots , g_k \in G$ such that the product of conjugates $ \left (g_1^{-1} g g_1 \right ) \cdots \left ( g_k^{-1} g g_k \right )$ is trivial.
 The same strategy for Corollary~\ref{Thm-BOPP} also demonstrates the following statement.

\begin{thm}[Corollary \ref{Thm-GTPP}]
 The existence of a generalized torsion element is not a profinite property.
\end{thm}

The paper is organized as follows. In Section~\ref{sec:prelim}, we briefly recall the definition and properties of orderable groups, profinite completions, fiber products, and related results.
The Magnus ordering and the Reidemeister--Schreier method are essential to construct our bi-orderable groups, so we add subsections introducing them in Section~\ref{sec:prelim}.
 We introduce the construction of a bi-orderable fiber product from a left-orderable group in Section~\ref{sec:proof1}.
 In the last section, Section~\ref{sec:proof2}, we show two main results, namely Corollary~\ref{Thm-BOPP} and Corollary~\ref{Thm-GTPP}.

\vspace{0.6cm} 

\noindent \textbf{Acknowledgement} 
 We are grateful to Hyungryul Baik and Sang-hyun Kim for helpful discussions and suggestions for qualitative improvement.
 The authors would like to thank Jan Kim for the valuable discussions that contributed to the proof of Lemma \ref{Core-Lem}.
 We thank Julian Wykowski for pointing out the missing references and useful comments.
 We are grateful to Paweł Piwek for careful reading and fruitful conversations.
 The authors would like to thank the anonymous referees for their suggestions to enrich the quality of this paper.
 Both authors have been supported by the National Research Foundation of Korea (NRF) grant funded by the Korea government (MSIT) RS-2025-00513595.
 The second author has also been supported by the Israel Science Foundation Grant 1576/23 (PI Nir Lazarovich).
% 
% 
% 

%%%%%%%%%%%%%%%%%%%%%%%%%%%%%%%%%%%%%%%%%%%%%%%%%%%%%%%%%%%%%%%%%%%%%%

\section{Preliminaries} \label{sec:prelim}
\subsection{Orderable groups}
 In this subsection, we briefly recall the definition of a bi-orderable group, related notions such as positive cones, and well-known results.

\begin{defn}
 Let $G$ be a group.
\begin{itemize}
    \item We say that $G$ is \emph{left-orderable} if there exists a (strict) total ordering $<$ on $G$ such that for every $g,x,y\in G$, $x<y$ implies $gx<gy$.
    \item We say that $G$ is \emph{bi-orderable} if there exists a (strict) total ordering $<$ on $G$ such that for every $g,h,x,y\in G$, $x<y$ implies $gxh<gyh$ .
\end{itemize}
\end{defn}

\begin{defn}
 Let $G$ be a group. A \emph{positive cone} $P$ is a subset of $G$ satisfying 
 \begin{itemize}
    \item for $x,y \in P$, $xy \in P$
    \item for any $g \in G$, exactly one of the following holds: either $g \in P$, $g^{-1} \in P$, or $g$ is the identity element.
 \end{itemize}
\end{defn}
 It is easy to show that $G$ is left-orderable if and only if there exists a positive cone $P \subset G$. If $G$ is left-orderable, define $P\coloneqq \{ g \in G : g > 1 \}$. Conversely, for a given positive cone $P$, declare $x>y$ when $y^{-1}x \in P$.
 Similarly, $G$ is bi-orderable if and only if there exists a conjugation invariant positive cone $P \subset G$, namely, a positive cone satisfying $gPg^{-1} = P$ for all $g \in G$.

 Left-orderable groups and bi-orderable groups share a lot of nice properties, for instance, these properties are closed under taking subgroups, direct products, and free products \cite{deroin2014groups}.
 However, it is well known that left-orderability is invariant under group extension, whereas bi-orderability is not. More precisely, from a short exact sequence
 $$ 1 \to K \to G \to H \to 1, $$
$G$ is left-orderable if $K$ and $H$ are left-orderable.
 In the realm of bi-orderable groups, such a statement is no longer true in general, but the following criterion is well known.

\begin{lem}\cite[Proposition 1.10]{clay2023orderable} \label{BO-ses}
 \ Suppose that we are given a short exact sequence $$ 1 \to K \overset{i}{\to} G \to H \to 1, $$ and $K$ and $H$ are bi-orderable. Let $P_K$ be the positive cone of $K$. If $gi(P_K)g^{-1} \subset i(P_K)$ for all $g\in G$, then $G$ is bi-orderable.
\end{lem}

For readers who are interested in more details about orderable groups, we refer to \cite{MR3560661} and \cite{deroin2014groups}.

\subsection{Magnus ordering}
 We dedicate this subsection to describing a specific bi-ordering on the free group $F_n$ using the Magnus expansion. 
 This bi-ordering plays a key role in our construction of a bi-orderable group, and we call this ordering the \emph{Magnus ordering}.
 We only recount the construction for the free group $F_\infty$ with countably infinite bases.
 It is straightforward to modify it for finitely generated free groups.
 We refer to Chapter 3.2 in \cite{MR3560661} for more details.

 Let $\{ x_1, x_2 ,  \cdots,x_n , \cdots \}$ be a free generating set for $F_\infty$ and $\ZZ[[X_n]]$ be the ring of power series with non-commuting variables $X_n$ for $n = 1,2,3, \cdots$. Fix a total ordering on the set $X \coloneqq \{ X_1, X_2, \cdots \}$ of variables.
 This total ordering can be extended to $X^*$, the (free) monoid generated by $X$ via the shortlex order. Namely, for two elements $m_1 \neq m_2 \in X^*$, we define $m_1 < m_2$ as follows:
 First, if the length of $m_1$ is strictly shorter than the length of $m_2$, then $m_1 < m_2$.
 If the lengths of $m_1$ and $m_2$ are the same, compare their first letters. If their first letters are different, since $X$ has a total ordering, we can determine whether $m_1 > m_2$ or $m_1 < m_2$.
 If the first letters are the same, then move to their second letters and compare them. 
 Since $m_1 \neq m_2$, these have different $n$-th letters for some $n$, and we can determine whether $m_1 > m_2$ or $m_1 < m_2$ using the total ordering of $X$.
 
From the total ordering on $X^*$, we can define an ordering on $\ZZ[[X_n]]$ as follows: Pick $f_1, f_2 \in \ZZ[[X_n]]$.
First, compare their constant terms. Since $f_1$ and $f_2$ has coefficients in $\ZZ$, either $c_1>c_2$, $c_1<c_2$ or $c_1=c_2$, where $c_i$ is the constant term of $f_i$ $(i=1,2)$.
If $c_1>c_2$, then we declare $f_1>f_2$. Similarly, $f_1<f_2$ when $c_1<c_2$.
When $c_1=c_2$, find the smallest (except for the constant terms) non-zero terms of $f_1$ and $f_2$, if they exist. 
By ``the smallest term,” we mean the one with respect to the total ordering on $X^*$.
If their smallest non-zero terms are different, the total ordering on $X^*$ determines the ordering between $f_1$ and $f_2$ as follows:
Suppose that $$f_1 = c + d_1 Y_1 + \cdots , \qquad  f_2 = c + d_2 Y_2 + \cdots,$$
where $d_1, d_2 \neq 0$ and $Y_1, Y_2 \in X^*$.
Here, $d_1Y_1$ and $d_2Y_2$ are the smallest terms (excluding the constant terms) of $f_1$ and $f_2$, respectively.
If $Y_1 \neq Y_2$, say $Y_1 < Y_2$, and we can express them as 
$$f_1 = c + d_1 Y_1 + \cdots , \qquad  f_2 = c + 0Y_1 + d_2 Y_2 + \cdots.$$
Thus we declare $f_1 > f_2$ if $d_1 > 0$, and $f_1<f_2$ if $d_1 < 0$. If $Y_1=Y_2$, then define 
$f_1 > f_2$ if $d_1 > d_2$, and $f_1 < f_2$ if $d_1 < d_2$.
When $Y_1=Y_2$ and $d_1=d_2$, then find their second smallest non-zero terms.
% When their smallest non-zero terms are the same, if the coefficients are different, we can also determine the order between them.
% Lastly, suppose that their coefficients are also the same.
Similar to the above, we can determine $f_1>f_2$ or $f_1<f_2$. Otherwise, let us move on to the third smallest non-zero terms.
Since $f_1 \neq f_2$, obviously some $n$-th smallest non-zero terms are different (also considering coefficients), and this explanation gives us an ordering $\ZZ[[X_n]]$.
We point out that this ordering is just a partial ordering (not a total ordering) since some power series cannot have the smallest non-zero terms (or $n$-th smallest non-zero terms).
For example, suppose that $\{ \ \cdots,X_{-2},X_{-1},X_0,X_1,X_2, \cdots \ \}$ is a set of variables. We define a total order on the set of variables by $$ X_i<X_j \iff i<j. $$
Then there are no smallest (except for the constant terms) non-zero terms of the power series $$ f:= 1 + \sum_{m \in \ZZ} X_i, \quad g:= 1 + \sum_{m \in 2 \ZZ} X_i - \sum_{m \in 2 \ZZ -1} X_i. $$
Therefore, the two power series $f$ and $g$ are not comparable with respect to the order.
Nonetheless, this obstruction does not frustrate us in constructing a bi-ordering on $F_{\infty}$, and we will explain why after we define a bi-ordering on the free group.

 Now we are ready to construct a bi-ordering on $F_{\infty}$. 
 Let $\left( \ZZ[[X_n]] \right)^\times$ be the group of units of the ring $\ZZ[[X_n]]$. 
 Fix a total ordering on $X\coloneqq \{ X_1 , X_2, X_3 , \cdots \}$ and consider the group homomorphism
 $i \colon  F_{\infty} \to \left( \ZZ[[X_n]] \right)^\times$ given by
 $$ x_n \mapsto 1+X_n ,\qquad  \ x_n^{-1} \mapsto 1-X_n+X_n^2-X_n^3+ \cdots. $$

 For $w_1 , w_2 \in F_{\infty}$, we declare $w_1 > w_2$ on $F_{\infty}$ if and only if $i(w_1) > i(w_2)$ on $\ZZ[[X_n]]$.
Then this ordering $<$ on $F_{\infty}$ is indeed a bi-ordering.
Since the length of $w_1$ is finite, only finitely many variables are in the power series $i(w_1) \in \ZZ[[X_n]]$. Thus, for each degree, $i(w_1)$ has only finitely many non-zero terms. This implies that the set of non-zero terms of $i(w_1)$ is actually well-ordered by the total ordering on $X^*$.
This argument can be applied to $w_2$, and we conclude that for given $w_1 \neq w_2 \in F_{\infty}$,
we can always determine whether $i(w_1)>i(w_2)$ or $i(w_1)<i(w_2)$ via this ordering. Thus, the ordering $<$ on $F_{\infty}$ is well-defined and a total ordering.
To check this, let $P_< \coloneqq \{ x \in F_{\infty} : x > 1 \}$. Clearly, $xy \in P_<$ when $x, y \in P_<$, and $F_{\infty} = P_< \sqcup P_<^{-1} \sqcup \{ 1 \}$.
 In addition, simple calculation implies that if $w \in P_<$, then for each $n$, we have
 \[ x_n w x_n^{-1} \in P_< ,\qquad \ x_n^{-1} w x_n \in  P_<. \]
 Therefore, the ordering $<$ on $F_{\infty}$ is a bi-ordering.

\subsection{Profinite completion and profinite properties} \label{Sebsec-PP}
 One of the main concepts in this paper is that of profinite properties.
 To introduce this notion, we begin with the definition of profinite groups.

\begin{defn}
 We say that a group $G$ is \emph{profinite} if $G$ is a topological group isomorphic to the inverse limit of an inverse system of discrete finite groups.
\end{defn}

 For any group $G$, there exists a related profinite group $\widehat{G}$, called the \emph{profinite completion} of $G$. We construct $\widehat{G}$ as follows. 
 Let $\mathcal{N}$ be the collection of all finite index normal subgroups in $G$. 
 For $N, M \in \mathcal{N}$, we give a reverse inclusion order, that is, $N \leq M$ if and only if $N \supset M.$ 
 For $N \leq M$, let $p_{NM} \colon  G/M \to G/N$ be the natural projection. 
Then $\mathcal{N}$ is a directed poset and the pair $\left ( \{ G/N \}_{N \in \mathcal{N}} , p_{NM} \right )$ is an inverse system.

\begin{defn}
 For a given group $G$, the pair $\left( \{ G/N \}_{N \in \mathcal{N}}, p_{NM} \right)$ is an inverse system. The inverse limit of this system, denoted by $\widehat{G}\coloneqq \varprojlim G/N$, is called the \emph{profinite completion} of $G$.
\end{defn}

\begin{rmk}
 There is the natural homomorphism $G$ to its profinite completion $\widehat{G}$ given by \[ g \mapsto \{ gN \}_{N \in \mathcal{N}}. \]
This homomorphism is not injective in general. In fact, the homomorphism is injective if and only if $G$ is residually finite.
\end{rmk}

Let $\mathcal{C}(G)$ be the set of isomorphism classes of all finite quotients of $G$.
As mentioned in Introduction, for finitely generated groups $G$ and $H$, $\widehat{G} \cong \widehat{H}$ if and only if $\mathcal{C}(G) = \mathcal{C}(H)$ (\cite{dixon1982profinite} and \cite{nikolov2007finitely}).
Indeed, the following correspondence is well known.

\begin{rmk} \label{first correspondence}
Let $\mathcal{N}(G)$ be the set of all finite index normal subgroups of $G$. When $G, \ H$ are finitely generated and $\mathcal{C}(G)=\mathcal{C}(H)$, we have $\mathcal{N}(G)=\mathcal{N}(H)$. 
Furthermore, we have a map $l:\mathcal{N}(G) \to \mathcal{N}(H)$ such that $G/N$ is isomorphic to $H/l(N)$ \cite{dixon1982profinite}.
\end{rmk}

We refer to \cite{reid2018profinite}, \cite{MR2599132}, and \cite{nikolov2007finitely} for readers who are interested in profinite completion and profinite groups.
Now we introduce the concept of a profinite property

\begin{defn}
 Let $\mathcal{P}$ be a group property of finitely generated residually finite groups. We say that $\mathcal{P}$ is a \emph{profinite property} if for any two finitely generated residually finite groups $G_1$ and $G_2$ with $\widehat{G_1} \cong \widehat{G_2}$, $G_1$ satisfies $\mathcal{P}$ if and only if $G_2$ satisfies $\mathcal{P}$.
\end{defn}

 In other words, if a group property $\mathcal{P}$ is a profinite property, then it can be detected by profinite completion, and equivalently, by taking finite quotients.

We close this subsection by recalling the well-known fact that profinite completions commute with direct products.

\begin{fact}
 Let $G$ and $H$ be groups. Then $\widehat{G \times H} \cong \widehat{G} \times \widehat{H}$.
\end{fact}

\subsection{Fiber products}
 In this subsection, we briefly give the definition of fiber product and introduce an interesting result obtained by Bridson \cite[Proposition 1.2]{MR4732948}.
 
\begin{defn}\label{def:fiberproduct}
 Let $p_1 \colon G_1 \to Q$, $p_2\colon G_2 \to Q$ be two epimorphisms between groups. The \emph{fiber product} of $p_1$ and $p_2$ is the subgroup given by
 $$ P\coloneqq \{ (g_1,g_2) : p_1(g_1)=p_2(g_2) \} < G_1 \times G_2 . $$
\end{defn}

 The following proposition is a key tool of our result and can be found in \cite[Proposition 1.2]{MR4732948}.
 The original statement and its proof can be found in \cite[Theorem 2.2]{MR3531665}, and this type of proposition was first suggested by Platonov and Tavgen in \cite{platonov1986grothendieck}.
 It has also played an important role in providing an example of Grothendieck pairs. See \cite{MR2119723}.

\begin{prop} \label{Fiber_profinite_completion}
Let $p_1 : G_1 \to Q$, $p_2 : G_2 \to Q$ be two epimorphisms of groups. If $G_1, G_2$ are finitely generated, $Q$ is finitely presented, $\widehat{Q}=1$, and $H_2(Q,\ZZ)=0$, then the inclusion $P \hookrightarrow G_1 \times G_2 $ induces an isomorphism of profinite completions $\widehat{P} \cong \widehat{G_1} \times \widehat{G_2}.$
\end{prop}

\subsection{Reidemeister--Schreier method}
 We recall the Reidemeister--Schreier method in this subsection.
For a given group presentation $G=\left< S \ | \ R \right>$ and a subgroup $H$, this method provides us a group presentation for $H$.
 For more details on the method, we refer to \cite[Chapter II.4]{MR577064}. We adopt the notation in \cite[p.5]{casals2024presentation}.

\begin{defn}[Reidemeister rewriting process]
 Let $G$ be a group generated by $ s_1 , \cdots , s_k $, $H$ be a subgroup of $G$, and $\overline{{}\cdot{}} \colon w \to \overline{w}$ be a right coset representative function for $G$ mod $H$.
 Suppose that a word $w \in H$ is expressed, as a generator for $G$, as 
 $$ w = s_{i_1}^{\epsilon_1} \cdots s_{i_l}^{\epsilon_l} $$
 where $\epsilon_i = \pm 1$.
 Then $w$ can also be expressed as follows:
$$ w = (\overline{1} s_{i_1}^{\epsilon_1} (\overline{1 s_{i_1}^{\epsilon_1}})^{-1} )(\overline{ s_{i_1}^{\epsilon_1}} s_{i_2}^{\epsilon_2}( \overline{ s_{i_1}^{\epsilon_1} s_{i_2}^{\epsilon_2}})^{-1})
(\overline{  s_{i_1}^{\epsilon_1}  s_{i_2}^{\epsilon_2} }  s_{i_3}^{\epsilon_3} (\overline{ s_{i_1}^{\epsilon_1}  s_{i_2}^{\epsilon_2} s_{i_3}^{\epsilon_3}} )^{-1}) 
\cdots ( \overline{s_{i_1}^{\epsilon_1} \cdots s_{i_{l-1}}^{\epsilon_{l-1}}} s_{i_l}^{\epsilon_l} (\overline{ s_{i_1}^{\epsilon_1} \cdots s_{i_l}^{\epsilon_l} })^{-1}) $$
 This expression of $w$ is called the \emph{Reidemeister rewriting process}, and we denote this expression by $\tau(w)$.
\end{defn}

\begin{thm}[Reidemeister--Schreier method] \label{RS method}
 Let $G=\left< S \ | \ R \right>$ be a group, and $H$ be a subgroup of $G$.
 Fix a set of right coset representatives $T$ of $H$ in $G$, and let $\overline{{}\cdot{}} \colon w \to \overline{w}$ be a right coset representative function.
 Then a group presentation of $H$ is given by
$$ H = \left< ts (\overline{ts})^{-1} | \ ts (\overline{ts})^{-1} = \tau(ts (\overline{ts})^{-1}) , \ \tau(trt^{-1})=1, \ t \in T, \ s \in S , \ r \in R  \right>.$$
\end{thm}

%%%%%%%%%%%%%%%%%%%%%%%%%%%%%%%%%%%%%%%%%%%%%%%%%%%%%%%%%%%%%%%%%%%%%%

\section{The construction of bi-orderable fiber products} \label{sec:proof1}
 We devote this section to proving the following result.

\begin{thm} \label{Thm-BO}
 Let $G$ be a non-trivial finitely generated left-orderable group. Say $G$ is generated by $n$ elements, and consider the following two epimorphisms 
 $$ \pi_1 : F_n*G \twoheadrightarrow G , \qquad \pi_2 : F_n \twoheadrightarrow G $$ where $\pi_2$ is the natural epimorphism, $\pi_1 |_{F_n} = \pi_2$, and $\pi_1|_{G}=\Id_G$, the identity map on $G$.
 Then the corresponding fiber product $P$ is bi-orderable.
\end{thm} 

 To establish this theorem, we need several lemmata. From now on, let $G,n,\pi_1,\pi_2$, and $P$ be as in Theorem~\ref{Thm-BO}.
 We also identify the elements of $F_n$ and $G$ with their canonical images in the free product $F_n * G$.
 We fix a free generating set $\{ s_i : 1 \leq i \leq n \}$ for $F_n$, and let $\widetilde{s_i}=\pi_2(s_i)$.
 Recall that the elements $\widetilde{s_i}$ form a generating set for $G$.

\begin{lem} \label{lem-free_generating}
$\ker(\pi_1)$ is freely generated by $S \coloneqq \{ g s_i \widetilde{s_i}^{-1} g^{-1} : g \in G , \ 1 \leq i \leq n \}$, thus $\ker(\pi_1)$ is isomorphic to $F_{\infty}$.
\end{lem}
\begin{proof}
 We use the Reidemeister--Schreier method (Theorem~\ref{RS method}) to show that $\ker(\pi_1)$ is freely generated by $$ S=\{ g s_i \widetilde{s_i}^{-1} g^{-1} : g \in G , \ 1 \leq i \leq n \}. $$
 First, we compute the group presentation for $\ker(\pi_1)$.
 To apply the method, we take $G$ as a right transverse set for $\ker(\pi_1)\trianglelefteq F_n * G$. 
 % This is possible since $\pi_1$ maps onto $G$ and sends a right coset of $\ker(\pi_1)$ in $F_n * G$ to its coset representative. 
%This is possible since $(F_n*G) / \ker (\pi_1) \cong G$ so $\ker(\pi_1)g \mapsto g$.
This is possible since $\pi_1$ maps onto $G$ and sends the right coset $\ker(\pi_1)g$ to $g$ for all $g \in G$ (i.e., $\ker(\pi_1)g \mapsto g$).
This also means that the right coset representative function $\overline{{}\cdot{}}\coloneqq F_n * G\to G$ is the same as the map $\pi_1$.
 
 Note that $F_n * G$ is generated by $\{s_1,\ldots,s_n,\widetilde{s_1},\ldots,\widetilde{s_n}\}$ and $\pi_1(s_i)=\widetilde{s_i}$. 
 So, the new generating set for $\ker(\pi_1)$ is given by
 $$ \left \{ y_{s_i,g} \coloneqq gs_i (\overline{gs_i})^{-1}= g s_i (g\widetilde{s_i})^{-1} , \ y_{\widetilde{s_i},g} \coloneqq g\widetilde{s_i} (\overline{g\widetilde{s_i}})^{-1}= g \widetilde{s_i} (g\widetilde{s_i})^{-1} : g \in G , \ 1 \leq i \leq n \right \}. $$
 Notice that $y_{s_i,g} = g s_i \widetilde{s_i}^{-1} g^{-1}$ and $y_{\widetilde{s_i},g}=1$.
 Thus, $y_{\widetilde{s_i},g}$ are redundant generators and we conclude that $\ker(\pi_1)$ is generated by 
 $\{ g s_i \widetilde{s_i}^{-1} g^{-1} : g \in G , \ 1 \leq i \leq n \}.$

 Next, we investigate the relations for $\ker(\pi_1)$.
 In the Reidemeister--Schreier method, there are two types of relations, so first we consider the relations of the form $ts(\overline{ts})^{-1} = \tau \left( \ ts(\overline{ts})^{-1} \ \right).$
 Recall that $ts(\overline{ts})^{-1}$ is the generator in the Reidemeister--Schreier method.
 So, we consider only $y_{s_i,g} = \tau(y_{s_i,g})$ since $y_{\widetilde{s_i},g}=1$.
 Let us write $g=\widetilde{s_{i_1}}^{\varepsilon_1}\widetilde{s_{i_2}}^{\varepsilon_2}\cdots \widetilde{s_{i_r}}^{\varepsilon_r}$ where $\varepsilon_i=\pm 1$. 
Then \begin{align*}
 y_{s_i,g} &= g s_i \widetilde{s_i}^{-1} g^{-1} \\
 &= \widetilde{s_{i_1}}^{\varepsilon_1}\widetilde{s_{i_2}}^{\varepsilon_2}\cdots \widetilde{s_{i_r}}^{\varepsilon_r} s_i \widetilde{s_i}^{-1} \widetilde{s_{i_r}}^{-\varepsilon_r}\cdots\widetilde{s_{i_2}}^{-\varepsilon_2}\widetilde{s_{i_1}}^{-\varepsilon_1}
\end{align*}
Now we compute $\tau(y_{s_i,g})$. To find this easily, we calculate the following formulae.
\begin{enumerate}
    \item $\alpha_k \coloneqq \overline{ \widetilde{s_{i_1}}^{\varepsilon_1}\widetilde{s_{i_2}}^{\varepsilon_2}\cdots \widetilde{s_{i_{k-1}}}^{\varepsilon_{k-1}}} \widetilde{s_{i_k}}^{\varepsilon_k} \left(\overline{ \widetilde{s_{i_1}}^{\varepsilon_1}\widetilde{s_{i_2}}^{\varepsilon_2}\cdots \widetilde{s_{i_{k-1}}}^{\varepsilon_{k-1}} \widetilde{s_{i_k}}^{\varepsilon_k}}\right) ^{-1},$ ($1 \leq k \leq r $) \\
    Since $\overline{\widetilde{s_{i}}^{\varepsilon_1}} = \widetilde{s_{i}}^{\varepsilon_1}$, $\alpha_k = 1$ for all $1 \leq k \leq r$.
    \item $\beta \coloneqq \overline{g}s_i \left(\overline{g s_i}\right)^{-1}$ \\
    Note that $\overline{g}=g$ since $g \in G$. Thus, we have $\beta = g s_i \widetilde{s_i}^{-1} g^{-1} = y_{s_i,g}$.
    \item $\gamma \coloneqq \overline{gs_i}\widetilde{s_i}^{-1} \left( \overline{gs_i\widetilde{s_i}^{-1}} \right)^{-1}$ \\
    $\gamma = g \widetilde{s_i} \widetilde{s_i}^{-1} \left(g \widetilde{s_i} \widetilde{s_i}^{-1} \right)^{-1} = g g^{-1} = 1$.
    \item $\delta_j \coloneqq \overline{gs_i\widetilde{s_i}^{-1} \widetilde{s_{i_r}}^{-\varepsilon_r}\cdots\widetilde{s_{i_{j+1}}}^{-\varepsilon_{j+1}} } \widetilde{s_{i_j}}^{-\varepsilon_j} \left(\overline{gs_i\widetilde{s_i}^{-1} \widetilde{s_{i_r}}^{-\varepsilon_r}\cdots\widetilde{s_{i_{j}}}^{-\varepsilon_{j}}}\right)^{-1}$ ($1 \leq j \leq r $) \\
    For each $j$, $\delta_j = g \widetilde{s_i} \widetilde{s_i}^{-1}\widetilde{s_{i_r}}^{-\varepsilon_r}\cdots\widetilde{s_{i_{j}}}^{-\varepsilon_{j}} \left(g \widetilde{s_i} \widetilde{s_i}^{-1}\widetilde{s_{i_r}}^{-\varepsilon_r}\cdots\widetilde{s_{i_{j}}}^{-\varepsilon_{j}}\right)^{-1} = 1. $
\end{enumerate}

 By the Reidemeister rewriting process, we have $\tau(y_{s_i,g}) = \alpha_1 \cdots \alpha_r \beta \gamma \delta_r \cdots \delta_1=y_{s_i,g}$.
So, the relations of the type $ts(\overline{ts})^{-1} = \tau(ts(\overline{ts})^{-1})$ are redundant and disappear in the group presentation.
 
 Consider the relations of the form $\tau(trt^{-1})=1$ for $t\in G$ and $r\in R$ the relations in $G$. 
 For each $r \in R$, $trt^{-1}$ can be written as a product of $\widetilde{s_i}$'s and their inverses.
 As we observed that each $\widetilde{s_i}$ is converted into $y_{\widetilde{s_i},g}$ by the rewriting process, each $\tau(trt^{-1})$ is represented by a product of $y_{\widetilde{s_i},g}$'s for some $g$ and their inverses. This implies that one can eliminate the relations of type $\tau (trt^{-1})=1$ from the presentation for $\ker(\pi_1)$ by removing redundant generators of the form $y_{\widetilde{s_i},g}$. Therefore, the group presentation for $\ker(\pi_1)$ is given by 
 $$ \left < g s_i \widetilde{s_i}^{-1} g^{-1} \ (g \in G, \ 1 \leq i \leq n ) \ | \ \emptyset \right > ,$$ 
 and so $\ker(\pi_1)$ is isomorphic to $F_{\infty}$ freely generated by $S$.
\end{proof}

\begin{lem} \label{P-fbf}
 $ P \cong F_{\infty} \rtimes F_n$.
\end{lem}
\begin{proof}
The proof is almost the same as the proof ($F_{\infty} \rtimes F_4$ part) of Theorem A in \cite{MR4732948}.
Recall that $\ker(\pi_1)$ is isomorphic to $F_{\infty}$ by Lemma \ref{lem-free_generating}.

The projection from $P$ to the second factor of $(F_n * G) \times F_n$ is surjective (recall that $\pi_1 |_{F_n} = \pi_2$.), so we have an epimorphism $f : P \twoheadrightarrow F_n$.
 Observe that $$\ker(f) = \{ (g_1,g_2) \in (F_n * G) \times F_n : \pi_1(g_1)=\pi_2(g_2), \ g_2=1 \}, \textnormal{ and }$$
 $$ \ker(\pi_1) = \{ g_1 \in F_n * G : \pi_1(g_1) = 1 \}.$$
Thus, these two subgroups are isomorphic, and the kernel of $f$ is $\ker(\pi_1) \cong F_{\infty}$. Hence, we obtain the following short exact sequence
$$ 1 \to F_{\infty} \hookrightarrow P \twoheadrightarrow F_n \to 1 $$ and the result immediately follows since this sequence is split.
\end{proof}

 For simplicity, let $x_{g,i}\coloneqq g s_i \widetilde{s_i}^{-1} g^{-1}$. Recall that $S=\{ x_{g,i} : g \in G , 1 \le i \leq n \}$ is a free generating set for $F_{\infty}$.

\begin{lem} \label{Calculation_s_x}
 In the expression $ P \cong F_{\infty} \rtimes F_n$, the homomorphism $F_n \to \Aut(F_{\infty})$ is given by $$ w \cdot (g s_j \widetilde{s_j}^{-1} g^{-1}) \coloneqq w (g s_j \widetilde{s_j}^{-1} g^{-1}) w^{-1}. $$
Then we have the following formulae.
\begin{enumerate}
    \item $ s_i \cdot x_{g,j} = x_{1,i} \ x_{\widetilde{s_i}g,j} \ x_{1,i}^{-1}$. 
    \item $ s_i \cdot x_{g,j}^{-1} = x_{1,i} \ x_{\widetilde{s_i}g,j}^{-1} \ x_{1,i}^{-1}$.
    \item $ s_i^{-1} \cdot x_{g,j} = x_{\widetilde{s_i}^{-1},i}^{-1} \ x_{\widetilde{s_i}^{-1}g,j} \  x_{\widetilde{s_i}^{-1},i}$.
    \item $ s_i^{-1} \cdot x_{g,j}^{-1} = x_{\widetilde{s_i}^{-1},i}^{-1} \ x_{\widetilde{s_i}^{-1}g,j}^{-1} \  x_{\widetilde{s_i}^{-1},i}$.
\end{enumerate}
\end{lem}
\begin{proof}
We only check the first and third items. 
% The remaining cases follow immediately from an elementary property of the group action on groups.
The remaining cases follow immediately from $(s_i^\pm \cdot x_{g,j})(s_i^\pm \cdot x_{g,j}^{-1}) = 1$.
From the semidirect product structure, we have 
$$  s_i \cdot x_{g,j} = s_i (g s_j \widetilde{s_j}^{-1} g^{-1}) s_i^{-1}. $$
%However, the last expression is not of the form in the free generating set $S$.
%So we need to modify it as follows:
To see that the right-hand side is a product of elements in the free generating set $S$, we need to modify it as follows:
\begin{align*}
 s_i (g s_j \widetilde{s_j}^{-1} g^{-1}) s_i^{-1} &= s_i (\widetilde{s_i}^{-1} \widetilde{s_i}) (g s_j \widetilde{s_j}^{-1} g^{-1}) (\widetilde{s_i}^{-1} \widetilde{s_i}) s_i^{-1} \\
 &= (s_i \widetilde{s_i}^{-1}) (\widetilde{s_i} g) s_j \widetilde{s_j}^{-1} (g^{-1} \widetilde{s_i}^{-1}) (\widetilde{s_i} s_i^{-1}) \\
 &= (s_i \widetilde{s_i}^{-1}) (\widetilde{s_i} g) s_j \widetilde{s_j}^{-1} (\widetilde{s_i} g)^{-1} (s_i \widetilde{s_i}^{-1})^{-1}.
 \end{align*}
Obviously, $s_i \widetilde{s_i}^{-1} = x_{1,i}$ is an element in the free generating set. Since $\widetilde{s_i} g \in G$, $(\widetilde{s_i} g) s_j \widetilde{s_j}^{-1} (\widetilde{s_i} g)^{-1}$ is in the free generating set. Thus, we conclude that
$$ (s_i \widetilde{s_i}^{-1}) (\widetilde{s_i} g) s_j \widetilde{s_j}^{-1} (\widetilde{s_i} g)^{-1} (s_i \widetilde{s_i}^{-1})^{-1} = x_{1,i} \ x_{\widetilde{s_i}g,j} \ x_{1,i}^{-1} $$ is the desired description.

 Next, we check the third one. Recall that $$ s_i^{-1} \cdot x_{g,j} = s_i^{-1} (g s_j \widetilde{s_j}^{-1} g^{-1}) s_i.$$
 Similar to the previous case, we obtain 
\begin{align*}
 s_i^{-1} (g s_j \widetilde{s_j}^{-1} g^{-1}) s_i &= s_i^{-1} (\widetilde{s_i} \widetilde{s_i}^{-1}) (g s_j \widetilde{s_j}^{-1} g^{-1}) (\widetilde{s_i} \widetilde{s_i}^{-1}) s_i \\
 &= (s_i^{-1} \widetilde{s_i}) (\widetilde{s_i}^{-1} g) s_j \widetilde{s_j}^{-1} (g^{-1} \widetilde{s_i}) (\widetilde{s_i}^{-1} s_i) \\
 &= (s_i^{-1} \widetilde{s_i}) x_{\widetilde{s_i}^{-1} g,j} (s_i^{-1} \widetilde{s_i})^{-1}.
 \end{align*}
Note that $$ s_i^{-1} \widetilde{s_i} = (\widetilde{s_i}^{-1} \widetilde{s_i}) s_i^{-1} \widetilde{s_i} = \widetilde{s_i}^{-1} (s_i \widetilde{s_i}^{-1})^{-1} \widetilde{s_i} = x_{\widetilde{s_i}^{-1},i}^{-1}, $$
 so we obtain the third formula.
\end{proof}

 Now we prove the main goal of this section.

\begin{proof}[Proof of Theorem \ref{Thm-BO}]
 Recall the short exact sequence $$ 1 \to F_{\infty} \hookrightarrow P \twoheadrightarrow F_n \to 1. $$
 Since free groups $F_{\infty}$ and $F_n$ are bi-orderable, it suffices to construct a conjugation invariant positive cone $P_{F_{\infty}} \subset F_{\infty}$ such that $w \cdot P_{F_{\infty}} = P_{F_{\infty}}$ for $w \in F_n$. 
 Recall that the action of $F_n$ on $F_{\infty}$ is given by the conjugation, so $w\cdot P_{F_{\infty}} = w P_{F_{\infty}} w^{-1}$.
 First, we give an order on $S$ by declaring 
 $$ x_{g,i} = g s_i \widetilde{s_i}^{-1} g^{-1} > g' s_j \widetilde{s_j}^{-1} g'^{-1} = x_{g',j} $$
 if and only if $ g > g'$, or $i>j$ whenever $g=g'.$
 Now we expand it to a bi-order $\prec$ on $F_{\infty}$ using the Magnus ordering.
 Note that $ x_{g,i} \succ 1 $ for any $g \in G$ and $1 \leq i \leq n$.
 More precisely, let $\ZZ[[X_{g,i}]]$ be the ring of power series with non-commuting variables $X_{g,i}$ for $g \in G , 1 \leq i \leq n$. 
 Consider the map $i : F_{\infty} \to \ZZ[[X_{g,i}]]$ given by
 $$ x_{g,i} \mapsto 1+X_{g,i} \ , \qquad x_{g,i}^{-1} \mapsto 1-X_{g,i}+X_{g,i}^2-X_{g,i}^3+ \ldots, $$
 with a total ordering $$ X_{g,i} > X_{g',j} $$
 if and only if $ g > g'$, or $i>j$ whenever $g=g'.$
 We define a bi-order $x \succ y$ on $F_{\infty}$ if and only if $i(x) > i(y)$ in $\ZZ[[X_{g,i}]]$.
 This construction is clearly a bi-ordering on $F_{\infty}$. 
 By Lemma \ref{BO-ses}, the only remaining part to show is that $w \cdot P_{\prec} = P_{\prec}$ for $w \in F_n$ where $P_{ \prec }$ is the positive cone of the bi-ordering $\prec$ on $F_{\infty}$.

Choose $a = x_{g_1,i_1}^{m_1} \cdots x_{g_k,i_k}^{m_k} \in P_{\prec}$. 
% Then $1<i(a)$ in $\ZZ[[X_{g,i}]]$.
% In other words, the coefficient of the smallest nonzero term of $i(a)-1$ is positive.
% Say, the smallest nonzero term is $X_{h_1,j_1} \cdots X_{h_l,j_l}$ and its coefficient is positive.
Then $1<i(a)$ in $\ZZ[[X_{g,i}]]$, and the coefficient of the smallest nonzero term of $i(a)-1$ is positive.
Suppose that the smallest nonzero term of $i(a)-1$ is $X_{h_1,j_1} \cdots X_{h_l,j_l}$.
Then 
\begin{align} \label{toPowerSeries}
 i(a) = 1 + C X_{h_1,j_1} \cdots X_{h_l,j_l} + \cdots
\end{align}
where $C>0$.
It suffices to show that $s_i \cdot a \in P_{\prec}$ and $\ s_i^{-1} \cdot a \in P_{\prec}$.

\begin{enumerate}
    \item $s_i \cdot a \in P_{\prec}$ \\
    From the first two formulae in Lemma \ref{Calculation_s_x}, we have
    $$ s_i \cdot a = x_{1,i} x_{\widetilde{s_i}g_1,i_1}^{m_1} \cdots x_{\widetilde{s_i}g_k,i_k}^{m_k} x_{1,i}^{-1}. $$
    Since $\prec$ is a bi-order on $F_{\infty}$, observe that $s_i \cdot a \in P_{\prec}$ if and only if $$ x_{\widetilde{s_i}g_1,i_1}^{m_1} \cdots x_{\widetilde{s_i}g_k,i_k}^{m_k} \in P_{\prec}. $$
    Consider $i(x_{\widetilde{s_i}g_1,i_1}^{m_1} \cdots x_{\widetilde{s_i}g_k,i_k}^{m_k}) \in \ZZ[[X_{g,i}]]$. Since $G$ is left-orderable, $g_{i_1}<g_{i_2}$ implies $\widetilde{s_i}g_{i_1} < \widetilde{s_i}g_{i_2} $.
    Hence, the smallest nonzero term of $i(x_{\widetilde{s_i}g_1,i_1}^{m_1} \cdots x_{\widetilde{s_i}g_k,i_k}^{m_k})-1$ is $X_{\widetilde{s_i}h_1,j_1} \cdots X_{\widetilde{s_i}h_l,j_l}$ and its coefficient is the same as $X_{h_1,j_1} \cdots X_{h_l,j_l}$.
    Recall Equation \eqref{toPowerSeries}.
    More precisely, this fact follows from 
    $$ i(x_{\widetilde{s_i}g_1,i_1}^{m_1} \cdots x_{\widetilde{s_i}g_k,i_k}^{m_k}) = 1 + C X_{\widetilde{s_i}h_1,j_1} \cdots X_{\widetilde{s_i}h_l,j_l} + \cdots .$$
    So, the coefficient of the smallest nonzero term (except for the constant term) is $C$, which is positive.
    This implies that $i(x_{\widetilde{s_i}g_1,i_1}^{m_1} \cdots x_{\widetilde{s_i}g_k,i_k}^{m_k}) > 1$.
    
    \item $s_i^{-1} \cdot a \in P_{\prec}$ \\
    This part is essentially the same as in the first case.
    From the last two formulae in Lemma \ref{Calculation_s_x}, we have
    $$ s_i^{-1} \cdot a = x_{\widetilde{s_i}^{-1},i}^{-1}
    x_{\widetilde{s_i}^{-1}g_1,i_1}^{m_1} \cdots x_{\widetilde{s_i}^{-1}g_k,i_k}^{m_k} 
    x_{\widetilde{s_i}^{-1},i} .$$
    We want to show $s_i^{-1} \cdot a \succ 1$. 
    Since $\prec$ is a bi-order on $F_{\infty}$,
    it is enough to check $$ x_{\widetilde{s_i}^{-1}g_1,i_1}^{m_1} \cdots x_{\widetilde{s_i}^{-1}g_k,i_k}^{m_k} \succ 1 . $$ 
    Consider $i( x_{\widetilde{s_i}^{-1}g_1,i_1}^{m_1} \cdots x_{\widetilde{s_i}^{-1}g_k,i_k}^{m_k} ) - 1 \in \ZZ[[X_{g,i}]] $. 
    Its smallest nonzero term is $X_{\widetilde{s_i}^{-1}h_1,j_1} \cdots X_{\widetilde{s_i}^{-1}h_l,j_l}$, and has a positive coefficient. 
    Hence, $s_i^{-1} \cdot a \in P_{\prec}$.
\end{enumerate}
 Therefore, $w \cdot a \in P_{\prec}$ for every $w \in F_n$, and $P \cong F_{\infty} \rtimes F_n$ is bi-orderable.
\end{proof}

\begin{rmk} \label{free-by-free BO}
    In general, a free-by-free group $F_{\infty}\rtimes F_n$ is not bi-orderable. 
    To construct an example, assume that $F_\infty$ is freely generated by $\{a_i\}_{i\in\mathbb{Z}}$ and $F_n$ is freely generated by $\{b_1,\ldots,b_n\}$. 
    Define a group homomorphism $\varphi:F_{n}\to \Aut(F_\infty)$ defined by $\varphi(b_j)(a_0)=a_j a_0^{-1} a_j^{-1}$ and $\varphi(b_j)(a_i)=a_i$ for $i\neq 0$. 
    Each $\varphi(b_j)$  is an automorphism of $F_\infty$, as it is a composition of transvections and an inversion. %\cite{MR1356145}
    Thus, we can define a free-by-free group $F_{\infty}\rtimes_{\varphi} F_n$ via $\varphi$. 
    This group contains an element $b_1 a_0 b_1^{-1} = a_1 a_0^{-1} a_1^{-1}$. However, both the left-hand side and right-hand side cannot simultaneously be positive or negative whenever a bi-order is defined on the semi-direct product. 
    Indeed, if $a_0>1$, then $b_1 a_0 b_1^{-1}>1$, but $a_1 a_0^{-1} a_1^{-1}<1$, which leads to a contradiction.

    As aforementioned, Bridson constructed a free-by-free group $M_{\Gamma}=F_{\infty}\rtimes F_4$ with an embedding $M_{\Gamma} \hookrightarrow (F_4 * \Gamma) \times F_4$ such that $ \widehat{M_{\Gamma}} \cong \widehat{F_4 * \Gamma} \times \widehat{F_4}$ \cite{MR4732948}. Since free-by-free groups may not be bi-orderable, it is unclear whether $M_{\Gamma}$ is bi-orderable. This explains why our main theorems cannot be obtained directly from Bridson's result.
\end{rmk}

\section{Proof of the main results} \label{sec:proof2}
 To establish the main results, we need the following technical ingredient.

\begin{lem} \label{Core-Lem}
 There exists a finitely presented group $G$ satisfying all of the following.
 \begin{itemize}
    \item $\widehat{G}=1$,
    \item $H_2(G,\ZZ)=0$,
    \item $G$ is left-orderable,
    \item $G$ contains a finitely generated, residually finite, and not bi-orderable subgroup $S$.
 \end{itemize}
\end{lem}
\begin{proof}
 Consider the following group 
 $$ G \coloneqq \left< a_1, a_2, a_3, a_4, b \ | \ 
 a_2^{-1}a_1a_2 = a_1^2, \ a_3^{-1}a_2a_3=a_2^{2}, \ a_4^{-1}a_3a_4=a_3^{2}, \ a_1^{-1}a_4a_1=a_4^{2}  , \ a_1^{-1} b^{2} a_1 = b^{3} \right>, $$ and 
 we claim that $G$ satisfies all the conditions.
Notice that $$G = H *_{a_1=t} \BS(2,3), $$ where
 $H$ is the Higman group, namely, $$ H = \left< a_1, a_2, a_3, a_4 \ | \ 
 a_2^{-1}a_1a_2 = a_1^2, \ a_3^{-1}a_2a_3=a_2^{2}, \ a_4^{-1}a_3a_4=a_3^{2}, \ a_1^{-1}a_4a_1=a_4^{2}
 \right>, $$ and
 $\BS(2,3) = \left< b,t \ | \ t^{-1} b^{2} t = b^{3} \right>$.
 
 Obviously, $G$ is finitely presented.
 To show $\widehat{G}=1$, we show that for any finite group $F$ and any group homomorphism $f:G \to F$, $f$ should be the trivial homomorphism.
 Recall that the Higman group $H$ is known to have no non-trivial finite quotients (see \cite{higman1951finitely}).
 Since the subgroup generated by $a_1, a_2, a_3, a_4$ of $G$ is isomorphic to the Higman group $H$, for any a group homomorphism $f:G \to F$ where $F$ is a finite group, we obtain $$ f(a_1)=f(a_2)=f(a_3)=f(a_4) = 1_F, $$
  where $1_F$ is the identity element in $F$.
From the relation $a_1^{-1} b^{2} a_1 = b^{3}$, we get $f(b)^{2}=f(b)^{3}$ since $f(a_1)=1_F$.
  This implies that $f(b)=1_F$, and hence, $f$ is the trivial homomorphism.
  
 The condition $H_2(G,\ZZ)=0$ is obtained from the fact that $G$ has a balanced presentation and $H_1(G,\ZZ)=0$. Use Lemma 4.2 in \cite{MR2119723}.

 %Note that $G$ is finitely generated and has the trivial abelianization.
 %Since a finitely generated bi-orderable group always has a surjective group homomorphism onto $\ZZ$, $G$ is not bi-orderable \cite[Theorem 2.19]{MR3560661}.
 To show that $G$ is left-orderable, we use the fact that the amalgamated product $A*_\ZZ B$ is left-orderable when $A$ and $B$ are left-orderable (\cite[Corollary 5.3]{Bludov2009Word}, \cite[Corollary 2.9]{clay2023orderable}).
 Since $H$ and $\BS(2,3)$ are left-orderable (\cite[Theorem A]{MR4053280}, \cite[Example 1.16]{clay2023orderable}), the desired property now follows.

 Now, the only remaining thing to prove is that $G$ contains a finitely generated, residually finite, and not bi-orderable subgroup $S$.
 From the construction, $G$ has $\BS(2,3)$ as a subgroup.
 Since $\BS(1,-1)$ embeds into $\BS(2,3)$ \cite[Proposition 7.11]{Levitt2015Quotients}, 
 $G$ has $\BS(1,-1)$ as a subgroup.
 Note that $\BS(1,-1)$ is finitely generated and residually finite \cite{MR285589}. As $\BS(1,-1)$ has a generalized torsion element, it is not bi-orderable \cite[Lemma 2.1 and Lemma 2.3]{Motegi2017Generalized}.
 Therefore, we complete the proof by taking $S=\BS(1,-1)$.
\end{proof}
\begin{rmk}
 For each positive integer $m$, consider 
 \begin{align*}
  G_m & \coloneqq H *_{a_1=t} \BS(2m,2m+1) \\
 &= \left< a_1, a_2, a_3, a_4, b \ | \ 
 a_2^{-1}a_1a_2 = a_1^2, \ a_3^{-1}a_2a_3=a_2^{2}, \ a_4^{-1}a_3a_4=a_3^{2}, \ a_1^{-1}a_4a_1=a_4^{2}  , \ a_1^{-1} b^{2m} a_1 = b^{2m+1} \right>.
 \end{align*}
 Note that $\BS(2m,2m+1)$ also contains $\BS(1,-1)$ \cite[Proposition 7.11]{Levitt2015Quotients}.
 Since replacing $\BS(2,3)$ with $\BS(2m,2m+1)$ does not affect the proof of Lemma \ref{Core-Lem}, 
 $G_m$ also satisfies the conditions in the Lemma.
\end{rmk}

\begin{rmk}
 The Higman group $H$ satisfies all the conditions in Lemma \ref{Core-Lem} except for the last condition.
 Obviously, $H$ is finitely presented.
 The condition $\widehat{H}=1$ follows from the fact that it has no non-trivial finite quotients \cite{higman1951finitely}. 
 It is proved that $H_2(H,\ZZ)=0$ \cite{MR2119723}, and $H$ is left-orderable \cite{MR4053280}.
 %However, $H$ is not bi-orderable since $H$ is finitely generated but $H_1(H,\ZZ)=0$.
 We do not know yet whether $H$ contains a finitely generated, residually finite, and not bi-orderable subgroup.
% If this question is affirmatively answered, ~~~~
 \end{rmk}

 Now we are ready to show the main theorem.

\begin{cor} \label{Thm-BOPP}
 Bi-orderability is not a profinite property.
\end{cor}
\begin{proof}
 Choose $G$ as in Lemma \ref{Core-Lem}, and let $S < G$ be a finitely generated, residually finite, and non-bi-orderable subgroup.
 Consider the fiber products $P$ of $\pi_1:F_n*G \to G$ and $\pi_2:F_n \to G$ and $Q$ of $q_1 \coloneqq \pi_1|_{F_n * S} \colon F_n * S \to G$ and $q_2 \coloneqq \pi_2 \colon F_n \to G$.
 Recall that since $S<G$, we get $q_1$ from $\pi_1$ by the restriction, and $q_1$ is surjective.
 In addition, $Q$ is a subgroup of $P$.
 By Proposition \ref{Fiber_profinite_completion}, we have an isomorphism  $$\widehat{Q} \cong \widehat{F_n * S} \times \widehat{F_n}. $$

 We assert that this isomorphism gives a concrete example to prove the theorem.
 Since $P$ is bi-orderable (Theorem \ref{Thm-BO}), so is $Q$.
  Since $S$ is residually finite, $(F_n * S)\times F_n$ is residually finite and hence, its subgroup $Q$ is residually finite.
 Lastly, $Q$ is finitely generated due to Lemma 1.1 in \cite{MR4732948}.
 
 Clearly, the right side group $(F_n * S) \times F_n$ is a finitely generated residually finite group.
 However, $(F_n * S) \times F_n$ is not bi-orderable.
 Therefore, we conclude that the bi-orderability is not a profinite property.
\end{proof}

\begin{rmk}
 One may wonder why we do not directly use the fiber product $P$ to conclude Corollary \ref{Thm-BOPP}.
 Indeed, we get an isomorphism  $\widehat{P} \cong \widehat{F_n * G} \times \widehat{F_n}$, and
 $P$ is bi-orderable but $(F_n*G) \times F_n$ is not bi-orderable.
 However, the group $(F_n*G) \times F_n$ is not residually finite, so from this isomorphism, we cannot conclude that bi-orderability is not a profinite property.
\end{rmk}

 From the proof of Corollary \ref{Thm-BOPP}, we can also deduce the following result by taking $S \cong \BS(1,-1)$.
 Note that $\BS(1,-1)$ has a generalized torsion element, whereas a bi-orderable group cannot have generalized torsion elements.

\begin{cor} \label{Thm-GTPP}
 The existence of a generalized torsion element is not a profinite property.
\end{cor}

%%%%%%%%%%%%%%%%%%%%%%%%%%%%%%%%%%%%%%%%%%%%%%%%%%%%%%%%%%%%%%%%%%%%%%

\begin{rmk}
One may ask whether left- or bi-circular orderability is a profinite property. 
First, these notions of circular orderability are strongly related to left- and bi-orderability, but they are not the main focus of this paper. 
We refer to \cite{MR3784820}, \cite{MR3813208}, \cite{MR3887426}, \cite{MR4275867} and \cite{MR4305778} for more on circular orderability.
Since every left-orderable group is left-circular orderable, it follows from Bridson's result \cite{MR4732948} that left-circular orderability is not a profinite property. 
Indeed, our contribution implies that bi-circular orderability is not a profinite property, 
since bi-orderability and bi-circular orderability are equivalent for torsion-free groups
% since bi-orderability and bi-circular orderability are equivalent when the given group is torsion-free 
\cite[Proposition 3.2]{MR4275867}.
\end{rmk}

\bibliography{preprint}
\bibliographystyle{abbrv}

\end{document}